\documentclass[a4paper,twoside,11pt,reqno]{amsart}
\usepackage{graphicx}
\usepackage{multicol}
\usepackage{nomencl}
\usepackage{geometry}
\usepackage{t1enc}	

\def\ncla#1{\index{#1}}
\makeindex
\usepackage[utf8]{inputenc}

\usepackage{amsmath}
\usepackage{amsfonts}
\usepackage{amssymb} 
\usepackage{color,euscript,graphicx,tabularx,epsfig,enumerate}
\usepackage{graphicx}
\usepackage{version}
\def\h1{\hspace{1cm}}
\def\h2{\hspace{2cm}}\def\h3{\hspace{3cm}}
\def\h4{\hspace{4cm}}\def\h5{\hspace{5cm}}
 \definecolor{grey}{rgb}{0.75,0.75,0.75}
\definecolor{orange}{rgb}{1.0,0.5,0.5}
\definecolor{brown}{rgb}{0.5,0.25,0.0}
\definecolor{pink}{rgb}{1.0,0.5,0.5}

\def\dis{\displaystyle}
 
\def\paragraph#1{\textit{#1}.}
\def\C{{\mathbb C}}

\newtheorem{tm}{Theorem}
\newtheorem{pp}{Proposition}

\newtheorem{lm}{Lemma}
\newtheorem{cl}{Corollary}
\newtheorem{defi}{Definition}
\newtheorem{req}{Remark}

\newcommand{\faitdisparaitre}[1]{}

 \def\cqfd{\unskip\kern 6pt\penalty 500
\raise -2pt\hbox{\vrule\vbox to5pt{\hrule width 4pt
\vfill\hrule}\vrule}}

\def\Ker{\textsf{\,Ker\,}}

 \def\dfl{\textsf{\,dfl\,}}

\def\A{\mathbf{A}}
\def\A2{{\mathbf{A}}^2}

\def\AB0{{\A}^2_{B( {0} ,\,{1})}}
\def\AB#1#2{{\A}^2_{B( {#1} ,\,{#2})}}

\def\D{\mathbf{D}}

\def\C{\mathbf{C}}

\def\E{\mathbf{E}}
\def\F{\mathbf{F}}

\def\mE{\mathcal{E}}

\def\binomial#1#2{{{#1}\choose{#2}}}

\def\monomials#1{\mathcal{M}}

\def\rank{\mathrm{rank}}

\def\eval{\, eval\, }
\def\schur{\, Schur\, }
\def\vect{\, vec\,}

\def\dis{\displaystyle}

\def\scr{\scriptstyle}

\title[Numerical approximation of multiple isolated roots]{Numerical approximation of multiple isolated roots of analytic
  systems}
\author{M.~Giusti}
\address{Marc Giusti \\Laboratoire LIX  \\ Campus de l'\'Ecole
  Polytechnique \\
1 rue Honor\'e d'Estienne d'Orves\\B\^atiment Alan
Turing\\CS35003\\91120 Palaiseau\\ France.} 
\email{Marc.Giusti@Polytechnique.fr} 
 \author{J.-C.~Yakoubsohn} 
\address{Jean-Claude Yakoubsohn \\ Institut de Math\'ematiques de
  Toulouse \\  Universit\'e Paul Sabatier \\ 118 route de Narbonne
  \\ 31062 Toulouse Cedex 9\\
France.}
\email{yak@mip.ups-tlse.fr} 
\date{Version of \today}
\begin{document}
\maketitle 
\tableofcontents
\vspace{0.cm}

\begin{abstract} 
We propose a numerical analysis of a simplified version
of the previous paper \emph{Multiplicity hunting and approximating
  multiple roots of polynomial systems} written by the two authors.
\end{abstract}
\pagebreak

\section*{Index of Main Symbols}
\def\indexspace{}
\begin{minipage}[b]{15cm}
  \begin{multicols}{3}
    \printindex
  \end{multicols}
\end{minipage}

%
\pagebreak

\section{Equivalent systems and multiplicity}
 The paper \textbf{\textsf{Multiplicity hunting and approximating
 multiple roots of polynomial systems}}~\cite{GY2} was written in
 a heuristic way. We achieve its numerical analysis in the present
 paper, by the way simplifying the procedure given previously.

\begin{defi}
A root $\zeta$
\ncla{$\zeta$} 
of an analytic system $f=0$
\ncla{$f$} 
(defined in a neighbourhood of
$\zeta$) is isolated and singular if
\begin{itemize}
 \item[1--] there exists a neighbourhood of $\zeta$ where
 $\zeta$ is the only root of $f=0$.
 \item[2--] the Jacobian matrix $Df(\zeta)$ is not full rank.
\end{itemize}
\end{defi}
Remark that the first assumption implies that the number of equations
is larger or equal than the number of variables. Note also that this
frame includes the important particular case of an analytic system
obtained by localizing a polynomial system.\\ 
We shall use equally the words singular or multiple for such a root.
We have explained in \cite{GY2} how to derive a regular system (i.e
admitting $\zeta$ as regular root) from a singular system at a multiple
isolated root, provided the assumption that $\zeta$ is exactly known. We
formalized this transformation by the notion of equivalent systems at
a point $\zeta$.  More precisely let $\zeta$ be a multiple isolated root of an
analytic system $f(x)=(f_1(x),\ldots,f_s(x))$ with $x$ in a
neighbourhood of $\zeta$ in $\C^n$ (note that $s \geq n$). Our
method computed a regular system admitting the same root $\zeta$, and that
we called \emph{equivalent}. Note that this is obtained
\textbf{\textsf{without adding new variables}} (important feature we
underline).\\


The \emph{multiplicity} of a root is an important numerical invariant.
In the case where there is only one variable and one equation, the
multiplicity of a root is exactly the number of derivatives which
vanishes at the root, which is unfortunately no longer true in the
multivariate situation. We have to introduce a more complicated
machinery.
 
Let us call 
\begin{itemize}
\item[1--]
$\C\{x-\zeta\}$ the ring of the germs of analytic functions at $\zeta$,
  i.e. the ring of convergent power series in a neighbourhood of
  $\zeta$, with maximal ideal generated by $x_1-\zeta_1,\ldots, x_n-\zeta_n$.
\item[2--] $I\C\{x-\zeta\}$ the ideal induced  generated by the
  ideal $I = I(f):=<~f_1,\ldots, f_s>$ in $\C\{x-\zeta\}$.
\end{itemize}

\begin{defi}
The multiplicity $\mu(\zeta)$\ncla{$\mu(\zeta)$} of an isolated root $\zeta$ is defined as the
dimension of the quotient space $\C\{x-\zeta\}/I\C\{x-\zeta\}$.
\end{defi}

Relatively to $<$ a admissible local order in $\C\{x-\zeta\}$, we denote
by $LT(IC(x-\zeta))$ the ideal generated by the leading terms of all elements of
$I\C\{x-\zeta\}$.
\begin{defi}
A (minimal) standard basis of $I\C\{x-\zeta\}$ is a finite set of series of
$I\C\{x-\zeta\}$ whose leading terms generate minimally $LT(IC(x-\zeta))$.
\end{defi}
We can prove that there is only a finite number of monomials, named
standard monomials, which are not in $I$. The following theorem is
classical in the literature about standard bases.
\begin{tm}\label{coxp178}
The following are equivalent:
 \begin{itemize}
 \item[1--] The root $\zeta$ is isolated.
 \item[2--] 
$dim\, \C \{x-\zeta\}/IC(x-\zeta)$ is finite.
 \item[3--] $dim\, \C \{x-\zeta\}/LT(IC(x-\zeta))$ is finite.
 \item[4--] There are only finitely many standard monomials.
 \end{itemize}
 Furthermore, when any of these conditions is satisfied, we have
 $$\mu(\zeta)=dim\, C \{x-\zeta\}/LT(IC(x-\zeta))= \textrm{number of standard
   monomials}.$$
\end{tm}
In the particular case of a localized polynomial system, whose
equation have a total degree upper bounded by some integer $d$, the
multiplicity is upper bounded by $d^n$.
\section{ Overview of this study}
To approximate a multiple isolated root is difficult because the root
can be a repulsive point for a fixed point method like the classical
Newton's method (see the example given by Griewank and Osborne in
~\cite{GO83}, p. 752). From a point of view of the theoretical
analysis, the technical background used when the derivative has
constant rank is not possible. This case is well understood and there
are many papers on this subject, see for
instance~\cite{xuli2008},~\cite{ArHi2011} and references within. To
overcome this drawback, the goal is to define an operator named
singular Newton operator generalizing the classical Newton operator
defined in the regular case. To do so we construct a finite sequence
of equivalent systems named \emph{deflation sequence}, where the
multiplicity of the root drops strictly between two successive
elements of the sequence. Hence the root is a regular root for the
last system. Then we extract from it a regular square system we named
deflated system. The singular Newton operator is defined as the
classical Newton operator associated to this deflated system.\\

We now explain the main idea of the construction of the deflation
sequence.  Since the Jacobian matrix is rank deficient at the root, it
means that there exists relations between the lines (respectively
columns) of this Jacobian matrix. These relations are given by the
Schur complement of the Jacobian matrix. When adding the
elements of the Schur complement to the initial system (we
call this operation \textit{kerneling}), we obtain an equivalent
system where the multiplicity of the root has dropped. In this way, a
sequence of equivalent system can be defined iteratively. This will be
explained in section \ref{Kern-Sing-Newton}.

Then we perform a local $\alpha$-theory of Smale of this singular
Newton operator. We first state a $\gamma$-theorem, i.e., a
result which gives the radius of quadratic convergence of this
singular Newton operator and next we give a condition using Rouché's
theorem to prove the existence of a singular root.

The context of this study is that of square integrable analytic
functions. In this way, it is possible to represent an analytic
function and its derivatives thanks to an efficient kernel : the
Bergman kernel. Moreover, our study is free of $\varepsilon$ (the measure
of the numerical approximation) in the following sense:
\begin{defi}
We said that a numerical algorithm is free of $\varepsilon$
if the input of the algorithm does not contain the variable $\varepsilon$.
\end{defi}  
The determination of a deflation sequence presented in the
table~\ref{dfl_table} is free of $\varepsilon$ under the assumption that
the norm defined in section~\ref{functional} (or an upper bound) is
given. To do that we present new results to determine by algorithms
free of $\varepsilon$:
\begin{itemize} 
\item[1--] The numerical rank of a matrix : this is achieved in
  section~\ref{rank_sec}.
\item[2--] How close to zero is the evaluation map, see the
  section~\ref{evaluation_sec}.
\end{itemize} 
We will see that the two previous problems are applications of the
$\alpha$-theory.\\

The analysis we present here generalizes what was done by Lecerf,
Salvy and the authors of the present work ~\cite{GLSY07}. Under the
hypothesis of a square system ($s=n$) and a multiple root of embedding
dimension one, i.e., where the the rank of Jacobian matrix drops
numerically by one, we treated the case of cluster of zeroes using
numerically the implicit function theorem. More precisely, there
exists an analytic function $\varphi(x_1,\ldots,x_{n-1})$ such that
$\zeta_n=\varphi(\zeta_1,\ldots,\zeta_{n-1})$ and hence $\zeta_n$ is a root of the
univariate function $h(x_n)=f_n(\varphi(x_1,\ldots,x_{n-1}),x_n)$.
Applying the results established in ~\cite{GLSY05} on the function
$h(x_n)$, we can deduce both the multiplicity of $\zeta_n$ and a way to
approximate quickly the root $\zeta_n$. Note that this work extends the
case of "simple double zeroes" previously studied by Dedieu and
Shub~\cite{DS00}.
\section{ Related works}
 
The case of one variable and one equation was hugely studied in the
literature and the generalization of the classical Newton operator is
the Schr\"oder operator defined in page 324
of~\cite{schroder1870}. Moreover, the $\alpha$-theory of this operator
is done in~\cite{GLSY05} with main references on this subject.\\
 
The multivariate case has been studied from purely symbolic and/or
numerical points of view. We will not discuss here the works with
only a symbolic treatment, see for instance~\cite{cox05}. One of
numerical pioneers is Rall~\cite{rall66}. He treats the particular
case where the singular root satisfies the following assumption: there
exists an index $m$, defined as the multiplicity of $\zeta$, such that
$N_m=\{0\}$ where $$N_1=\Ker Df(\zeta),\quad N_{k+1}=N_k\cup \Ker
Df^{k+1}(\zeta),\quad k=1:m-1.$$ Then it is possible to construct
iteratively an operator to retrieve the local quadratic convergence of
the classical Newton operator.  The idea of the construction of this
operator consists to project iteratively the error $x_0-\zeta$ on the
kernels $N_k$ and its orthogonal $N_k^\perp$.\\ 

At the same time, the idea to use a variant of a Gauss-Newton's method
to approximate a singular isolated root has been investigated by
Shamanskii in ~\cite{shamanskii67}. But this method converges
quadratically towards the singular root under very particular
assumptions.\\

Another techniques are bordered techniques, where some assumption is
done on the root. For instance, if the operator induced by the
projection from $\Ker Df(\zeta)$ into $\Ker (Df(\zeta)^*)^\perp$:
$$\dis~\pi_{(\Ker Df(\zeta)^*)^\perp}D^2f(\zeta)(z,\pi_{\Ker Df(\zeta)})$$
is invertible, then the $(\zeta,0)$ is a regular root of a system, called
bordered system, having $2n-r$ variables. The bordered system is
constructed from the initial system and from the singular value
decomposition of the Jacobian matrix. This way has been developed by
Shen and Ypma in~\cite{ShenYpma2005} and extends this bordered
technique used by Griewank~\cite{G85} in the case of deficient rank
one. At the beginning of the eighties a collection of papers addresses
the problem of the approximation of the singular roots with similar
techniques ~\cite{reddien1978},~\cite{reddien1979}, ~\cite{DK180},
~\cite{DK280}, ~\cite{GO81}, ~\cite{DKK83} ~\cite{kelleysuresh1983},
~\cite{yamamoto1983}. These methods previously cited are purely
numerical methods and neither the geometry of the problem nor the notion of
multiplicity are mentioned.\\
  
Ojika in ~\cite{ojika87} proposes a similar method called deflation
method to compute a regular system from the singular initial one, by
mixing both symbolic and numerical calculations. This paper is an
extension of an algorithm previously developed in~\cite{OWM83}. The
search of a regular system deals with Gauss forward elimination but there
is no analysis of this procedure, especially no numerical determination
of the rank. Note also that the attempt to classify the singular
roots suffers from not being related to the concept of multiplicity.
Moreover, there is no study of complexity, in the case where we study
a localized polynomial system. This approach was echoed
by Lecerf in~\cite{lecerf02}. He was able to give a deflation
algorithm which outputs a regular triangular system at a root
$\zeta$. Moreover he studied precisely the complexity of his deflation
algorithm, which is in:
  $$\mathcal{O}\left (n^3(nL+n^\Omega)\mu(\zeta)^2\log(n\;\mu(\zeta))\right )$$ where $n$
is the number of variables, $\mu(\zeta)$ the multiplicity, $3\le \Omega<4$ and
$L$ is the length of the straight line program describing the system.\\

Leykin, Verschelde and Zhao proposed in~\cite{lvz06} a similar
modified deflation method, based on the following observation: if the
numerical rank of the system is $r$, there exists an isolated solution
$(\zeta,\delta)\in \C^n\times\C^{r+1}$ of the system
  \begin{equation}\label{dfl_versch}
  Df(x)B\delta=0,\quad \delta^*h-1=0,
\end{equation}  
where $B\in \C^{n\times (r+1)}$ and $h\in\C^{r+1}$ are randomly
chosen. The multiplicity of the root $(\zeta,\delta)$ of the deflated
system is lower than the multiplicity $\mu(\zeta)$ of the root $\zeta$ of the initial
system. Then a step of deflation consists to add the
equations~(\ref{dfl_versch}) to the initial system. The theorem is
then that it is enough to perform $\mu(\zeta)-1$ steps of deflation to get a
regular system. This implies that the numbers of variables and
equations can double in the worst case. And unfortunately the
determination of the numerical rank, based on the work of
~\cite{fierro_hansen_05}, is not free of $\varepsilon$.\\

In the same same vein we have the papers of Dayton and
Zeng~\cite{DZ05} which treats the polynomial case, Dayton, Li and Zeng in the
analytic case~\cite{DLZ11} and Nan Li, Lihong Zhi ~\cite{li2014}. Particular cases were studied by Nan Li and Lihong Zhi in several papers~\cite{li121}, ~\cite{li2012}. But all these papers furnish a superficial
numerical analysis of their algorithms.\\ 

The duality and the relationship with the Macaulay matrices constitute
the theoretical background of Mourrain~\cite{Mou97}, Mantzaflais and
Mourrain~\cite{MM11} or more recently Hausenstein, Mourrain, Szanto
in~\cite{hauenstein_mourrain_szanto_16}. Actually, the relations
between the columns (respectively the lines) represent those of the
space (respectively, columns). As we shall point out, a classical fact
show that all these relations can be found through the Schur complement.

 \section{ Tracking the rank of a matrix}\label{rank_sec}
 Let $s \geq n$ be two integers, $M$ a $s\times n$-matrix with complex
 coefficients, $U\Sigma V^*$ a singular value decomposition of $M$,
 and $\sigma_1\ge \ldots \ge \sigma_n$ its singular values.\\

We consider the elementary symmetric sums of the $\sigma_i$'s, i.e.:
$$s_{k}=\sum_{1\le i_1<\ldots <i_k\le n}
\sigma_{i_1}\ldots\sigma_{i_k},\quad k=1:n.$$
\ncla{$s_k$}
In other words, the singular values are the roots of the polynomial
$s(\lambda)$ of
degree $n$ 

$$ s(\lambda) = \prod_{i=1}^n (\lambda - \sigma_i) = \lambda ^n + \sum_{1 \le i \le n} (-1)^{(n-i)} s_{n-i} \lambda ^i .$$
\ncla{$s(\lambda)$}
By convention $s_0=1$. Let us remark that this convention is
natural: it allows to treat the case where all the singular values
are zero, which means that the matrix $M$ is
null and its rank is zero.\\

More generally if the rank of $M$ is $r$, the $s_i$'s are
non-zero up to the rank ($i=0:r$), and zero after. Then for $k=n-r:n$,
the quantities $s_{n-k}$ are non-zero and we can introduce:

\begin{itemize}
\item[1--] $\dis b_{k}(M):=\max_{0\le i\le k-1}
\left (\frac{s_{n-i}}{s_{n-k}}\right )^{\frac{1}{k-i}}$.\ncla{$b_k(M)$}
\vspace{0.2cm}

\item[2--] $\dis g_{k}(M):=\max_{k+1\le i\le n} 
\left (\frac{s_{n-i}}{s_{n-k}}\right )^{\frac{1}{i-k}}$.\ncla{$g_k(M)$}
\vspace{0.3cm}

\item[3--] $\dis a_{k}(M):=b_{k}(M)\, g_{k}(M).$\ncla{$a_k(M)$}
\end{itemize}
with the convention $g_n(M)=1$.\\

We precise the notion of $\varepsilon$-rank used in the sequel.

\begin{defi}\label{defi_numerical_rank}

Let $\varepsilon$ be a nonnegative number. A matrix $M$ has
$\varepsilon$-rank\ncla{$\varepsilon$-rank} equal to $r_{\varepsilon}$ if its singular values verify
\begin{equation}\label{test_rank}
\sigma_1\ge \ldots \ge \sigma_{r_{\varepsilon}}>\varepsilon\ge \sigma_{r_{\varepsilon}+1}
\ge\ldots \ge \sigma_n.
\end{equation}
\end{defi}
Observe that an upper bound for the $\varepsilon$-rank is the rank $r$
itself.\\
Let $\Sigma_\varepsilon$ the matrix obtained from $\Sigma$ by 
putting $\sigma_ {r+1}=\ldots= \sigma_n=0$. We define
$M_\varepsilon=U\Sigma_\varepsilon V^*$.
\begin{req}
If $rank\, M\ge r$, we know that $M_\varepsilon$
\ncla{$M_\varepsilon$} is the nearest matrix of
$M$ which is of rank $r$.
\end{req}

\begin{req}
The definition~\ref{defi_numerical_rank} is justified by
the Eckardt-Young-Mirsky theorem which has a long story
in low rank approximation theory: see \cite{EY36}, \cite{mirsky60} and
\cite{markovsky2011} for more recent developments.
\end{req}
For simplicity let us denote by $a_k$, $b_k$, $g_k$ the corresponding
values $a_k(M)$, $b_k(M)$, $g_k(M)$.
\begin{tm}\label{rank_tm}
Let a matrix $M$ be such that $\rank(M)=r$.
Let $m$ an integer be such that $n-r \le m \le n$,
and $$\varepsilon=\frac{3a_{m}+1-\sqrt{(3a_{m} +1)^2-16a_{m}}}{4g_{m}}.$$
If $a_{m}<1/9$ then the matrix $M$ has $\varepsilon$-rank equal to $n-m$. 
\end{tm}
\begin{proof} As $n-r \le m$, the quantity $s_{n-m}$ is not zero since
  it is positive. Let us consider the polynomials
$$p(\lambda)=\frac{1}{s_{n-m}}s(\lambda)=\frac{1}{s_{n-m}}\prod_{i=1 }^n(\lambda-\sigma_i)=
\sum_{i=0}^{n}(-1)^{n-i}\frac{s_{n-i}}{s_{n-m} }\lambda^{i}$$
\ncla{$p(\lambda)$}
and
$$q(\lambda)= \sum_{i=m}^{n}(-1)^{n-i}\frac{s_{n-i}}{s_{n-m}}\lambda^{i}.$$
\ncla{$q(\lambda)$}
\begin{lm}\label{q}
Let $\tau:= g_m |\lambda|$.
\ncla{$\tau$} Then for all $\lambda$ such that $|\lambda| <
1/g_m$, hence for all $\tau<1$:
$$ |q(\lambda)| \ge |\lambda|^m \frac{1-2\tau}{1-\tau} $$
\end{lm}
\begin{proof}
\begin{align}
|q(\lambda)|&=\left |
\lambda^m+ \sum_{i=m+1}^{n}(-1)^{n-i}\frac{s_{n-i}}{s_{n-m}}\lambda^{i}
\right |\nonumber
 \\&\ge 
 |\lambda|^m- \sum_{i=m+1}^{n}\frac{s_{n-i}}{s_{n-m} } |\lambda|^{i} \nonumber
 \\
 &
\ge |\lambda|^m\left (1- \sum_{i=m+1}^{n}\frac{s_{n-i}}{s_{n-m} } |\lambda|^{i-m}\right )
 \nonumber
 \\
 &
\ge |\lambda|^m\left (1- \sum_{i\ge m+1 } (g_m |\lambda|)^{i-m} \right )
\nonumber\\
&
\ge |\lambda|^m\frac{1-2g_m|\lambda|}{1-g_m|\lambda|}. 
\end{align}
\end{proof}
We first prove that $0$ is the only root of $q(\lambda)$
in the open ball $\dis B\left (0,\frac{1}{2g_m}\right )$.
Let $\nu$ be a non-zero root of $q(\lambda)$. 
Then we have by lemma \ref{q} $$0=q(\nu)=|q(\nu)| \ge |\nu|^m \frac{1-2g_m|\nu|}{1-g_m|\nu|}.$$
Hence $\dis |\nu | \ge \frac{1}{2g_m}$.
\\
Now consider the trinomial 
\begin{equation}\label{pol2}
2\tau^2-(3a_m+1)\tau+2a_m.
\end{equation}
If $a_m <1/9$,
then this trinomial has two real roots $\tau_1 <\tau_2$, since its 
$$\Delta = (3a_m+1)^2 - 16a_m = 9a_m{^2} - 10a_m +1 = (9a_m -
1)(a_m-1)$$ 
is positive. We can check explicitely that $\tau_1$ is positive, since it
boils down to $a_m$ being positive.

We prove that for $|\lambda|$ satisfying $\dis\frac{\tau_1}{g_m} \le |\lambda| < \frac{1}{2g_m}$,
$p(\lambda)$ has $m$ roots counting with multiplicities in the open
ball $B(0,|\lambda|)$ (note that the range of the interval where $|\lambda|$ is
asked to live is positive, since $\tau_1 < 1/2$).
To do that, we verify that Rouché's inequality
\begin{equation}\label{eq_rouche_rank}
|p(\lambda)-q(\lambda)|<|q(\lambda)|
\end{equation}
 holds on the sphere of radius $|\lambda|$.
We have
 \begin{align}\label{eq_rouche_rank_g}
 |p(\lambda)-q(\lambda)|&\le 
 \sum_{i=0 }^{m-1}\frac{s_{n-i}}{s_{n-m} }|\lambda|^{i}
 \nonumber
 \\
 &
\le \sum_{i=0 }^{m-1} b_m^{m-i}|\lambda|^{i}
\nonumber\\
 &
 \le |\lambda|^m\frac{b_m/|\lambda|}{1-b_m/|\lambda|}
\nonumber \\
 &
 \le \frac{a_m}{g_m|\lambda|-a_m}|\lambda|^m.
 \end{align}
 We check that $\dis
 \tau-a_m > \tau_1 - a_m = \frac{-a_m +1
   -\sqrt{\Delta}}{4}$ is positive if $a_m<1/9$.\\
From ~(\ref{eq_rouche_rank_g}) and lemma ~\ref{q}, we see that the
 Rouch\'e's inequality is satisfied if
 $$\frac{a_m}{\tau-a_m}|\lambda|^m
< \frac{1-2\tau}{1-\tau}|\lambda|^m.$$
Since $|\lambda|$, $1-\tau$ and $\tau-a_m$ are positive, this is equivalent to
the trinomial (\ref{pol2}) being negative, which is insured by the
condition $a_m<1/9$.\\
Hence under the condition $a_m<1/9$ the polynomial
$p(\lambda)$ has exactly $m$ roots counting the multiplicities in the open ball $B(0,|\lambda|)$ where
$$ \varepsilon:=\frac{\tau_1}{g_m}\le |\lambda| < \frac{1}{2g_m}.$$
Consequently we have
$$\sigma_1\ge\ldots \ge \sigma_{n-m}>\varepsilon\ge \sigma_{n-m+1}\ge \ldots\ge \sigma_n.$$
We are done.
\end{proof}
 \begin{tm}
 The algorithm of the table~\ref{table_rank} computes the
 $\varepsilon$-rank of a matrix thanks to the theorem~\ref{rank_tm}. 
 \end{tm}
\begin{req}
In fact this algorithm is free of $\varepsilon$ and we call the computed
$\varepsilon$-rank the \emph{numerical} rank of the given matrix.
\end{req}
  \begin{table}
 $$\fbox{
 \begin{minipage}{1\textwidth }
$$
\textbf{ numerical rank }$$
\begin{enumerate}[1-]
\item \quad Input : a matrix $M\in \C^{s\times n}$, $s\ge n$.
\item \quad  Compute the singular values of $M$ : $\sigma_1\ge\ldots
  \ge\sigma_n$.
\item \quad Let $r$ be the rank of $M$, i.e. $\sigma_{r+1} > 0,
  \sigma_r = 0$.
\item  \quad From these $\sigma_i$'s, compute the quantities $a_k$, $
  k=n-r:n$ and $g_k$ defined in the section~\ref{rank_sec}.
\item \quad if there exists $m \ge n-r$ s.t. $a_m<1/9$ then
\\
\item \quad \quad\quad $\dis \varepsilon:=\frac{3a_m+1-\sqrt{(3a_m+1)^2-16a_m}}{4g_m}$
\item \quad \quad\quad the $\varepsilon$-rank of the matrix $M$ is $n-m$.  \quad\quad\textsf{from the theorem~\ref{rank_tm}}
\item \quad else
\item \quad \quad\quad $\varepsilon<\sigma_n$. The $\varepsilon$-rank of the matrix $M$ is $n$.
\item  \quad end if
\item \quad Output : the $\varepsilon$-rank of the matrix $M$.
\end{enumerate}
   \end{minipage}
 }
 $$
 \caption{}\label{table_rank}
 \end{table} 
\section{ The functional framework}\label{functional}
Let $n\ge 2$, $R_\omega\ge0$\ncla{$R_\omega$}
 and $\omega\in \C^n$. \ncla{$\omega$}  We consider the set
$\A2(\omega,R_{\omega})$
\ncla{$\A2(\omega,R_{\omega})$} of the square integrable analytic functions in the
open ball $B(\omega,R_{\omega})$, which is an Hilbert space equipped with the
inner product
$$<f,g> =\int_{B(\omega,R_{\omega})}f(z)\overline{g(z)} d\nu(z),$$
where $\nu$ is the Lebesgue  measure on $\C^n$, normalized so that
$\nu(B(\omega,R_{\omega}))=R_{\omega}^{2n}$.  
\\
Next $(\A2(\omega,R_{\omega}))^s$ has an hilbertian structure with
the inner product
$$<f,g>=\sum_{i=1}^s<f_i,g_i>.$$
We denote by $||f||$\ncla{$||f||$} the associated norm.\\

Observe that this framework includes the case of an analytic system
obtained by localizing a polynomial system.

\subsection{The Bergman kernel}
in~\cite{rudin08} and S.G. Krantz in~\cite{krantz13}. Since for each
$x\in B(\omega,R_{\omega})$ and $f\in\A2(\omega,R_{\omega})$, the
evaluation map $f\mapsto f(x)$ is a continuous linear functional
$eval_x$ on $\A2$, there exists from the Riesz representation theorem
an element $h_x \in \A2$ such that
$$f(x)=eval_x(f)=<f,h_x>.$$\\
Set down the function $\rho := x \mapsto \rho_x = \|x-\omega \|$.
\ncla{$\rho_x$}
\begin{defi}\label{bk}
The function $(z,x) \mapsto H(z,x):=\overline{h_x(z)}$\ncla{$H(z,x)$} is named the
Bergman kernel. It has the reproducing property :
$$f(x)=\int_{B(\omega,R_{\omega})}f(z)\, H(z,x)\,d\nu(z),\quad \forall
f\in \A2(\omega,R_{\omega}).$$
We say that the Bergman kernel reproduces $\A2(\omega,R_{\omega})$. We state
some classical properties of this reproducing kernel. 
\end{defi}
\subsection{Properties}
\begin{pp}\label{berg_ker}
\quad
\begin{itemize}
\item[1--] $\dis H(z,x)=
\frac{R_{\omega}^{2}}{(R_{\omega}^2-<z-\omega,x-\omega>)^{n+1}}$
\\
\item[2--]  $ \dis H(x,x)=\|H(\bullet,x)\|^2=
\frac{R_{\omega}^{2}}{(R_{\omega}^2-\|x-\omega\|^2)^{n+1}}=
\frac{R_{\omega}^{2}}{(R_{\omega}^2-\rho_x^2)^{n+1}}.$
\\
\item[3--] For all $f\in \A2(\omega,R_{\omega})$ we have
$$\dis |f(x)=|\int_{B(\omega,R_{\omega})}f(z)H(z,x)d\nu(z)|
\le \frac{\|f\|\, R_{\omega}}{(R_{\omega}^2-\rho_x^2)^\frac{n+1}{2}}$$
\end{itemize}
\end{pp}
\begin{proof}
See Theorem 3.1.3. page 37 in~\cite{rudin08}.
\end{proof}
 The previous proposition generalizes to higher derivatives.
  \begin{pp}\label{DkFH}
\quad Let $k\ge 0$, $\omega\in\C^n$, $x\in B(\omega,R_{\omega})$ and $u_i\in
\C^n$, $i=1:k$.  Let us introduce
$$H_k(z,x,u_1,\ldots, u_k)=
\frac{(n+1)\cdots (n+k)<z-\omega,u_1 >\cdots
<z-\omega, u_{k}>}{(R_{\omega}^2-<z-\omega,x-\omega>)^{k}}H(z,x).$$
We have
\begin{itemize}
\item[1--] $\dis D^{k}f(x)(u_1,\cdots,u_k)= 
\int_{B(\omega,R_{\omega})}f(z)\, H_k(z,x,u_1,\cdots, u_k)\, d\nu(z).$ 
\\\\
\item[2--]  $\dis
\|D^kf(x)\|\le 
||f||\, 
\frac{(n+1)\cdots (n+k)\, R_{\omega}^{1+k}}{(R_{\omega}^2-\rho_x^2)^{\frac{n+1}{2}+k}}$
\end{itemize}
(evidently if $k=0$ the range where $i$ lives is empty, and the
products\\
~$(n+1)\cdots (n+k)$ and $<z-\omega,u_1 >\cdots <z-\omega, u_{k}>$ are $1$.)
\end{pp}
To prove this we need the following 
\begin{lm}\label{HK} 
$$\|H_k(\bullet,x,u_1,\ldots,u_n)\|
\le
 \frac{(n+1)\ldots (n+k)\,R_{\omega}^{1+k}}{(R_{\omega}^2-\rho_x^2)^{ \frac{n+1}{2}+k}}\|u_1\|\ldots \|u_k\|.$$
\end{lm}
\begin{proof}
We have to compute the integral of $H_k\bar H_k$ on the ball $B(\omega,R_{\omega})$. This is reduced to estimate
$$I_k=\dis \int_{B(\omega,R_{\omega})}\frac{R_{\omega}^{2}}{(R_{\omega}^2-<z-\omega,x-\omega>)^{n+1+k}(R_{\omega}^2-\overline{<z-\omega,x-\omega>})^{n+1+k}}d\nu(z)$$
since
\begin{align*}
\|H_k(z,x,u_1,\ldots,u_n)\|&\le 
(n+1)\ldots(n+k)\|u_1\|\ldots \|u_k\|\, R_{\omega}^{1+k} \, I_k^{1/2}.
\end{align*}
We have 
\begin{align*}
\dis I_k&=\int_{B(\omega,R_{\omega})}
H(z,x)
\frac{1}{(R_{\omega}^2-<z-\omega,x-\omega>)^{k}(R_{\omega}^2-\overline{<z-\omega,x-\omega>})^{n+1+k}}d\nu(z)
\\
&= 
\frac{1}{(R_{\omega}^2-\rho_x^2)^{n+1+2k}}
\end{align*} 
using the formula for the Bergman kernel (Proposition \ref{berg_ker})
  and its reproducing property applied to the function
$\dis\frac{1}{(R_{\omega}^2-\rho_x^2)^{n+1+2k}}.$

The proof of the lemma follows.
\end{proof}
 We now prove the proposition~\ref{DkFH}.
 \begin{proof}
We proceed by induction. The proposition~\ref{berg_ker} treats the
case $k=0$. Next, we have:
\begin{align*}
D^{k+1}f(x)&(u_1,\ldots,u_k,u_{k+1})=\left .
\frac{d}{dt}D^kf(x+tu_{k+1})(u_1,\ldots,u_k)\right |_{t=0}
\\
&=\left .\frac{d}{dt}\int_{B(\omega,R_{\omega})}f(z) H_k(z,x+tu_{k+1},u_1,\ldots,u_k)d\nu(z)\, \right |_{t=0} 
\\
&=\int_{B(\omega,R_{\omega})}f(z)\frac{ H_k(z,x,u_1,\ldots ,u_k)(n+1+k) 
<z-\omega, u_{k+1}>}{(R_{\omega}^2-<z-\omega,x-\omega>)}d\nu(z)
\\
&=\int_{B(\omega,R_{\omega})}f(z) H_{k+1}(z,x,u_1,\ldots,  u_{k+1}) d\nu(z).
\end{align*}
Hence the first assertion holds.
For the second assertion, we write
$$\|D^kf(x)(u_1,\ldots,u_k)\|\le \| f\|\,\|H_k(\bullet,x,u_1,\ldots ,u_k)\|.$$
Using the lemma ~\ref{HK}, we are done. 
\end{proof}
From the propositions~\ref{berg_ker} and ~\ref{DkFH}
we deduce easily the following
\begin{pp}\label{Berg_F_DFK}
For all $k\ge 0$, $x\in\C^n$ and $f\in (\A2(\omega,R_{\omega}))^s$ we have
\begin{equation*}
\dis \|D^kf(x)\|\le ||f||\, 
\frac{(n+1)\ldots (n+k)R_{\omega}^{1+k}}{(R_{\omega}^2-\rho_x^2)^{\frac{n+1}{2}+k}}.
\end{equation*}
\end{pp}
\section{ Analysis of the evaluation map}\label{evaluation_sec}
 The evaluation map is defined by
$$\eval\,:\, (f,x)\mapsto eval_x(f) = f(x)$$ from $(\A2(\omega,R_{\omega}))^s\times
B(\omega,R_{\omega})$ to $\C^s$.\\
\ncla{$eval_x$}
Let  $\dis c_0:=\sum_{k \ge 0}(1/2)^{2^k-1}$ ($\sim 1.63...$), and
$\alpha_0$ ($\sim 0.13...$) be the first positive root of  the trinomial
$(1-4u+2u^2)^2-2u$.
\ncla{$c_0$}
\ncla{$\alpha_0$}
\\
We study the question: when the value $f(x)$ can be considered as
small? We give a precise meaning of being small 
without the use of any $\varepsilon$.
\begin{tm}\label{evaluation_map_tm}
Let $f=(f_1, \ldots, f_s)\in \A2(\omega,R_{\omega})^s.$
Let $x\in B(\omega,R_{\omega})$ and $\|x-\omega\|=\rho_x$.
If 
$$\frac{c_0}{R_{\omega}}\,(R_{\omega}^2-\rho_x^2)^\frac{n+1}{2}\|f(x)\|+ \rho_x <R_{\omega}$$
 and
$$
 \frac{(n+1)(n+2)}{2}\,(R_{\omega}^2-\rho_x^2)^{(n-3)/2} \left (
\|f\|\,R_{\omega}+(R_{\omega}^2-\rho_x^2)
 \right )\, \|f(x)\|\le \alpha_0
 $$
   then $f(x)$ is small at the following sense : the Newton sequence defined by 
   $$(f^0,x_0)=(f,x), \quad (f^{k+1  },x_{k+1})=((f^{k},x_{k})-
 D\eval(f^{k},x_{k})^\dagger  \eval(f^{k},x_{k})),           \quad k\ge 0,$$
 converges quadratically towards a certain
  $\dis (g,y)\in   
  (\A2(\omega,R_{\omega}))^s\times B(\omega,R_{\omega})$ satisfying $g(y)~=~0$.
  More precisely we have 
 $$( \|f-g\|+\|x-y\|^2)^{1/2}\le \frac{c_0}{  R_{\omega}}\,(R_{\omega}^2-\rho_x^2)^\frac{n+1}{2}\|f(x)\|.$$
\end{tm}
 In a straightforward way, we get the corollary 
\begin{cl}\label{cl_evaluation}
Let us consider $x=\omega$ in the theorem~\ref{evaluation_map_tm}.
If  		
$$
c_0R_{\omega}^{n-1}\|f(x)\|<1$$
and
$$
\frac{(n+1)(n+2)}{2}\,R_{\omega}^{n-2}  
\left ( \|f\|+R_{\omega}\right )\, \|f(x)\|\le \alpha_0.
 $$
  then $f(x)$ is small.
More precisely
  there exists $(g,y)\in (\A2(x,R_{\omega}))^s\times B(x,R_{\omega})$ such that $g(y)=0$ and
 $$( \|f-g\|+\|x-y\|^2)^{1/2}\le c_0\,R_{\omega}^n\|f(x)\|.$$
\end{cl}
\subsection{Estimates about the derivatives of the evaluation map}
\begin{pp}\label{Deval_pseudo}
$$\| D\eval (f,x)^\dagger\|\le 
\frac{1}{R_{\omega}}(R_{\omega}^2-\rho_x^2)^{\frac{n+1}{2}}.$$
\end{pp}
\begin{proof}
The derivative of the evaluation map is given by
$$D\eval(f,x)(g,y)=g(x)+Df(x)y.$$
Hence $(g,y)\in \ker D\eval(f,x)$ iff $g(x)+Df(x)y=0$.
That is 
$$ <g_i,H(\bullet,x)>+<y,Df_i(x)^*>=0,\quad i=1:s.$$
 In term of inner product in $(\A2)^s\times \C^n$ we have
$$<g,(0,\ldots,0,H(\bullet,x),0,\ldots,0)>+<y,Df_i(x)^*>=0,
\quad i=1:s.$$
This shows that the vector space $(\ker D\eval (f,x))^\perp$
is generated by the set of
$$(H(\bullet ,x)v, Df(x)^*v)$$
where $v\in\C^n$. The condition
$$D\eval (f,x)(H(\bullet,x),\,Df(x)^*v)=u$$
becomes
$$\left(H(x,x)I_{s}+Df(x)Df(x)^*\right )v=u.$$
The matrix $\mE=H(x,x)I_{s}+Df(x)Df(x)^*$ is the sum of a diagonal positive matrix and an hermitian matrix. By Weyl theorem (page 203 in G.W. Stewart, J.Q. Sun, Matrix Perturbation Theory, Academic Press, 1990) the eigenvalues of the matrix $\mE$ are greater than those of $H(x,x)I_{s}>0$.
Hence the norm of the inverse matrix $\mE^{-1}$ satisfies 
$$\|\mE^{-1}\|\le \frac{1}{H(x,x)}.$$
This permits to calculate $\|D\eval (f,x)^\dagger\|$. In fact, let
$u,v\in \C^n$ be such that $\mE v=u$. We have
\begin{align*}
\|D\eval (f,x)^\dagger u\|^2&=\|H(\bullet,x)\|^2\,\|v\|^2+\|Df(x)v\|^2
\\
&=H(x,x)\, \|v\|^2+\|Df(x)^*v\|^2.
\end{align*}
Since the matrix $\mE^{-1}$ is hermitian, we can write
\begin{align*}
\|D\eval (f,x)^\dagger u\|^2&=
v^*\mE v
\\
&=u^*\mE^{-1} u
\\
&\le \|\mE^{-1}\|\, \|u\|^2.
\end{align*}
Finally
\begin{align*}
\|D\eval (f,x)^\dagger \|^2 &\le \|\mE^{-1}\|
\\
&
\le \frac{1}{H(x,x)}
\\
&\le\frac{1}{R_{\omega}^2} \left(R_{\omega}^2-\rho_x^2\right )^{n+1}.
\end{align*}
This proves the proposition.
\end{proof}	
\begin{pp}\label{Dkeval}
$$\|D^k\eval (f,x)\|\le 
\frac{(n+1)\ldots (n+k)\,\|f\|\,R_{\omega}^{1+k}}{(R_{\omega}^2-\rho_x^2)^{\frac{n+1}{2}+k}}
+\frac{k(n+1)\ldots (n+k-1)\,R_{\omega}^{k}}{(R_{\omega}^2-\rho_x^2)^{\frac{n+1}{2}+k-1}}.
$$
\end{pp}
\begin{proof}
We have
\begin{align*}
D^k\eval (f,x)&(g^{(1)},y^{(1)},\ldots ,g^{(k)},y^{(k)})
\\ &= D^kf(x)(y^{(1)},\ldots,y^{(k)})+\sum_{j=1}^k D^{k-1}g^{(j)}(x)(y^{(1)},\ldots,\widehat{y^{(j)}},\ldots, y^{(k)}),
\end{align*}
where $\widehat{y^{(j)}}$ signifies that this term does not appear.
Then using the proposition~\ref{DkFH} we find that
\begin{align*}
\|D^k&\eval (f,x)(g^{(1)},y^{(1)},\ldots ,g^{(k)},y^{(k)})\|
\\ &\le \| D^kf(x)(y^{(1)},\ldots,y^{(k)})\|+\sum_{j=1}^k\| D^{k-1}g^{(j)}(x)(y^{(1)},\ldots,\widehat{y^{(j)}},\ldots, y^{(k)})\|
\\
&
\le 
\frac{(n+1)\ldots (n+k)\,\|f\|\,R_{\omega}^{1+k}}{(R_{\omega}^2-\rho_x^2)^{\frac{n+1}{2}+k}}\|y^{(1)}\|\ldots \|y^{(k)}\|
\\
&\quad\quad\quad 
 +\sum_{j=1}^k\frac{(n+1)\ldots (n+k-1)\,\|g^{(j)}\|\,R_{\omega}^{k}}{(R_{\omega}^2-\rho_x^2)^{\frac{n+1}{2}+k-1}}\|y^{(1)}\|\ldots \widehat{\|y^{(j)}\|}\ldots\|y^{(k)}\|
.
\end{align*}
We bound $\|y^{(j)}\|$ and $\|g^{(j)}\|$ by $\|(g^{(j)},y^{(j)})\|$. We obtain
\begin{align*}
\|D^k&\eval (f,x)(g^{(1)},y^{(1)},\ldots ,g^{(k)},y^{(k)})\|
\\ & \le 
\left (\frac{(n+1)\ldots (n+k)\,\|f\|\,R_{\omega}^{1+k}}{(R_{\omega}^2-\rho_x^2)^{\frac{n+1}{2}+k}} 
+ \frac{k(n+1)\ldots (n+k-1)\,\|g^{(j)}\|\,R_{\omega}^{k}}{(R_{\omega}^2-\rho_x^2)^{\frac{n+1}{2}+k-1}}\right )\\
& 
\hspace{8cm}||(g^{(1)},y^{(1)})\|\ldots\|(g^{(k)},y^{(k)})\|
.
\end{align*}
Finally
$$\|D^k\eval (f,x)\|\le 
\frac{(n+1)\ldots (n+k)\,\|f\|\,R_{\omega}^{1+k}}{(R_{\omega}^2-\rho_x^2)^{\frac{n+1}{2}+k}}
+\frac{k(n+1)\ldots (n+k-1)\,R_{\omega}^{k}}{(R_{\omega}^2-\rho_x^2)^{\frac{n+1}{2}+k-1}}
.$$
\end{proof}
\subsection{Proof of the theorem~\ref{evaluation_map_tm}}
The proof uses the theorem 128 page 121 in J.-P. Dedieu, Points fixes, z\'eros et la
m\'ethode de Newton. Springer, 2006.
\begin{tm}\label{dedieu_128}
Let $f$ an analytic map from $\E$ to $\F$ two Hilbert spaces 
be given. Let  $x\in \C^n$. We suppose that $Df(x)$ is surjective. We introduce the quantities
\begin{itemize}
\item[1--] $\dis \beta(f,x)=\|Df(x)^\dagger f(x)\|. $
\vspace{.2cm}

\item[2--] $\dis \gamma(f,x)=\sup_{k \ge 2}\|\frac{1}{k!}
Df(x)^\dagger D^kf(x)\|^{\frac{1}{k-1}}.$
\vspace{.2cm}

\item[3--] $\dis \alpha(f,x)= \beta(f,x)\gamma(f,x)$.
\end{itemize}
 Let $\alpha_0$ and $c_0$ be the constants introduced in this section.
\\
If $\alpha(f,x)\le \alpha_0 $ then there exists a zero $\zeta$ of $f$ in the ball $B(x_0,c_0\beta(f,x_0))$ and the Newton sequence
$$x_0=x,\quad x_{k+1}=x_k-Df(x_k)^\dagger f(x_k),\quad k\ge 0,$$
converges quadratically towards $\zeta$. 
\end{tm}
We are now ready to prove the theorem~\ref{evaluation_map_tm}.
\begin{proof}
 The proof consists to verify the condition $\alpha(\eval,(f,x))\le\alpha_0$.
Using the propositions ~\ref{Deval_pseudo} and ~\ref{Dkeval}, we
are able to bound the quantity $\gamma(\eval,(f,x))$. We obtain
\begin{align*}
\gamma(\eval,(f,x))&\le \sup_{k\ge 2} \left (\frac{1}{k!}
\, \|D\eval (f,x)^\dagger\|\, \|D^k\eval (f,x)\|\right )^{\frac{1}{k-1}}
\\
&\le
\sup_{k\ge 2} \left (\binomial{n+k}{k}\frac{\|f\|\, R_{\omega}^{k}}{(R_{\omega}^2-\rho_x^2)^{k}}
+
\binomial{n+k-1}{k-1}\frac{R_{\omega}^{k-1}}{(R_{\omega}^2- \rho_x^2)^{k-1}} 
\right )^{\frac{1}{k-1}}.
 \end{align*}
 We know that $\dis \binomial{n+k}{k}=
 \frac{n+k}{k}\binomial{n+k-1}{k-1}$. Moreover
 the function $\dis k\mapsto \binomial{n+k}{k}^{\frac{1}{k-1}}$ decreases. Hence $\dis \binomial{n+k}{k}^{\frac{1}{k-1}}\le \frac{(n+1)(n+2)}{2}$. Then we get the following  point estimate
 \begin{equation}\label{gamma_eval_eq}
 \gamma(\eval,(f,x))\le
 \frac{(n+1)(n+2)R_{\omega}}{2(R_{\omega}^2-\rho_x^2)
}
 \left (
 \frac{\|f\|\, R_{\omega}}{(R_{\omega}^2-\rho_x^2)
} +  1
\right )
 .
 \end{equation} 
In the same way the quantity $\alpha(\eval, (f,x))$ can be bounded by
  \begin{align*}
 \alpha(\eval,(f,x))&\le \gamma(\eval, (f,x))\, \beta(\eval,(f,x))
 \\
 &\le \gamma(\eval, (f,x))\, \|D\eval (f,x)^\dagger\|\,\|f(x)\|
 \end{align*}
 Using the inequalities of propositions ~\ref{Deval_pseudo} and (\ref{gamma_eval_eq})
 we get
\vspace{0.2cm}

 \begin{align}\label{alpha_eval_eq}
\alpha&(\eval , (f,x)) \le \frac{(n+1)(n+2)}{2}\,(R_{\omega}^2-\rho_x^2)^{(n-3)/2} \left (
\|f\|\,R_{\omega}+(R_{\omega}^2-\rho_x^2)
 \right )\, \|f(x)\|.
\end{align}
The condition $$\frac{(n+1)(n+2)}{2}\,(R_{\omega}^2-\rho_x^2)^{(n-3)/2} \left (
\|f\|R_{\omega}+(R_{\omega}^2-\rho_x^2)
 \right )\, \|f(x)\|\le \alpha_0$$
 implies evidently $\alpha(\eval (f,x))\le \alpha_0$.
 \\
 Hence the theorem~\ref{dedieu_128} applies. The Newton sequence
 $$(f^0,x_0)=(f,x), \quad (f^{k+1  },x_{k+1})=((f^{k},x_{k})-
 D\eval(f^{k},x_{k})^\dagger  \eval(f^{k},x_{k}),           \quad k\ge 0,$$
 is convergent towards a certain
  $\dis (g,y)\in  B((f,x),c_0\beta(\eval,(f,x))\subset
  (\A2(\omega,R_{\omega})^s\times \C^n$.
 That is to say
\begin{align*}
(\|f-g\|^2+\|x-y\|^2)^\frac{1}{2} &\le c_0\beta(\eval,(f,x))
\\
&
\le 
c_0 \|D\eval (f,x)^\dagger\|\,   \|f(x)\|
\\
&\le\frac{c_0}{ R_{\omega}}\, (R_{\omega}^2-\rho_x^2)^\frac{n+1}{2}\|f(x)\|.
\end{align*}
This implies that $y\in B(\omega,R_{\omega})$ since we have
 \begin{align*}
 \|y-\omega\|&\le \| y-x\|+\rho_x
 \\
 &\le \frac{c_0} {R_{\omega}}\,(R_{\omega}^2-\rho_x^2)^\frac{n+1}{2}\|f(x)\|+\rho_x
 \\
 &<R_{\omega}. \quad\quad \textsf{from assumption}.
\end{align*} 
We are done.
 \end{proof}
\section{ Kerneling and singular Newton operator}\label{Kern-Sing-Newton}
It consists to prepare the system by dividing the generators into two
families. The invariant leading to this partition is the rank $r$ of
the Jacobian matrix $Df(\zeta)$ which is not maximal since $\zeta$ is
singular. Without loss of generality we can assume that the first $r$
generators have linearly independent affine parts.
\\
Since the notion of Schur complement is intensively used in the sequel, we remember its definition.
\begin{defi} 
 The Schur complement of a matrix $\dis M=\left (\begin{array}{cc}
A&B \\C&D \end{array}\right )$ of rank $r>0$ associated to an invertible
 submatrix $A$ of rank $r$ is by definition 
  $\schur(M):=D-CA^{-1}B$. 
  \\
  If $r =0$ we define $\schur(M):=M$.\ncla{$\schur(M)$}
  \end{defi}
 We also note by $\vect (\bullet)$ the operator which transforms a
 matrix into a line vector by concatenating its lines.
\begin{defi}\label{defi_Kf}
Let $\varepsilon\ge 0$, $0\le r<n$ and $f=(f_1,\ldots,
f_s)\in\C\{x-x_0\}^s$.  Let us suppose $D_{1:r}f_{1:r}(x_0)$ has an
$\varepsilon$-rank equal to $r$.
We define the kerneling operator 
$$ K\, :\,f\mapsto \left (f_1,\ldots,f_r, vec(\schur (Df(x)) )\right
)\in \C\{x-x_0\}^{r+(n-r)\, (s-r)}.
$$
We say that $K(f)$ is an $\varepsilon$-kerneling  of  $f$
if we have
\begin{equation}\label{test_Kf}
\|K(f)\|\le \varepsilon.
\end{equation}
We say that the kerneling is exact when $\varepsilon=0$.\ncla{$K(f)$}
\end{defi}  
\begin{defi}\label{dfl_seq} \textbf{$(${\bf Deflation sequence}$)$.}
Let $\varepsilon\ge 0$, $x_0\in \C^n$ and $f=(f_1,\ldots, f_s)\in
\C\{x-x_0\}^s.$ The sequence
\begin{align*}
  F_0&=f\\
  F_{k+1}&=K ( F_k ), \quad k\ge 0.
 \end{align*}
 is named the  deflation sequence.
 \\ 
 The thickness is the index $\dis \ell$ \ncla{$\ell$}
 where the $\varepsilon$-rank of
 $DF_\ell(x_0)$ is equal to $n$, and not before.
\\
We name \emph{deflation system} $\dfl(f)$ of $f$\ncla{$\dfl(f)$}  a system of rank $n$
extracted from $F_\ell$.
 \end{defi}
We adopt the term \emph{thickness} which is the translation of the
french word \emph{épaisseur} introduced by Ensalem in~\cite{emsalem78}
rather than the term \emph{depth} more recently used by Mourrain,
Matzaflaris in ~\cite{MM11} or Dayton, Li, Zeng ~\cite{DZ05},
~\cite{DLZ11}.  We shall see in section \ref{Mult-drops-kerneling}
that the thickness is finite. ~\,~$\circ$
\begin{tm}\label{dfl_tm}
Let $x_0\in \C^n$ and $f\in\A2(x_0,R_{\omega})$. Then the algorithm
described in the table~\ref{dfl_table} proves the existence of a
deflation sequence where the tests verifying the inequalities
~\ref{defi_numerical_rank} and ~\ref{defi_Kf} are performed
respectively thanks to the theorem~\ref{rank_tm} and the
corollary~\ref{cl_evaluation}.
\end{tm}
\begin{table}
 $$\fbox{
 \begin{minipage}{1\textwidth }
\textbf{deflation sequence and deflated system}
\begin{enumerate}[1-]
\item \quad  Input : $x_0\in \C^n$, $f\in\A2(x_0,R_{x_0})$
\item \quad $\dfl(f)= \{\emptyset\}$
\item \quad $F:=f$.
\item \quad $\dis \eta:=\frac{2\alpha_0}{(n+1)(n+2)(R_{x_0}+\|F\|)R_{x_0}^{n-2}}$
\\
\item \quad if $\|F(x_0\|\le \eta$ then \quad\quad\quad
  \textsf{test justified by corollary~\ref{cl_evaluation}}
 \item  \quad\quad $r:=\textbf{numerical rank }(DF(x_0))$ 
\item \quad \quad if $r<n$ then 
\item \quad  \quad \quad $F:=K(F)$
\item \quad  \quad\quad go to $2$
\item \quad\quad else
\item \quad\quad \vbox{$\dfl(f)$ a deflated system of numerical rank $n$
      extracted from $F$}
\item \quad\quad end if
\item \quad end if
\item \quad Output : $\dfl(f)$.
\end{enumerate}
    \end{minipage}
 }
 $$
   \caption{}\label{dfl_table}
 \end{table}
 \begin{defi} 
The classical Newton operator associated to the deflation system
$dfl(f)$ of $\varepsilon$-rank $n$ is named the singular Newton operator
of the initial system $f$.
\end{defi}
Rather than to compute the deflation sequence introduced in the
definition ~\ref{dfl_seq}, it is sufficient to start from a truncated
deflation sequence.  To do that we need the following definition.
\begin{defi}\label{dfl_tr}
Let $p\ge 1$. We note by $Tr_{x_0,p}(F)$ the truncated series at the
order $p$ of the analytic function $F$ at $x_0$.\\
We name the truncated deflation sequence at the order $p$ at $x_0$ the
sequence~:
\begin{align*}
T_0&=Tr_{x_0,p}(f)
\\
T_{k+1}&=Tr_{x_0,p-k-1}\left (K\,(T_k)\,
\right ),\,\quad 0\le k\le p.\end{align*} 
\end{defi}
  To define the singular Newton operator
it is sufficient to know the thickness of the deflation sequence
of the definition~\ref{dfl_seq}. From this knowledge the
determination of the singular Newton operator will use the truncated deflation sequence at the order of the thickness, say
 $\ell$, i.e. that  is to say  that the rank of $F_\ell$ is full.
 \begin{pp}
 Let $\eta>0$ and $x_0\in\C^n$.
 Let $\ell$ the thickness of the deflation sequence of the definition~\ref{dfl_seq}. Let us consider the truncated deflation sequence $(T_k)_{k\ge 0}$ at the order $\ell +1$ at $x_0$ of the definition~\ref{dfl_tr}.
 Then the singular Newton operator associated to $f$ is equal to the Newton operator associated to $T_{\ell}$. 
 \end{pp}
 \begin{proof}
 Since $T_0$ is the truncated series at the order $\ell$
of $F_0$, from construction it is easy to see
 that for all $k=0:\ell$, $T_k$ is the truncated series of $F_k$
at the order $p-k$.  The conclusion of the proposition follows.
 \end{proof}
 \begin{table}
 $$\fbox{
 \begin{minipage}{1\textwidth }
$$ \textbf{singular Newton}$$
\begin{enumerate}[1-]
\item \quad  Input : $x_0\in \C^n$, $f\in\A2(x_0,R_{x_0})
\\
$\item \quad $    \dfl(f)=\textsf{deflated system}(f)$.
 \\
\item \quad Output : If $\dis \dfl(f)\ne \emptyset$ then $N_{\dfl(f)}(x_0)$ else $x_0$.\ncla{$N_{\dfl(f)}$}
\end{enumerate}
    \end{minipage}
 }
 $$
 \caption{}\label{Ndflx0}
 \end{table}
\section{ The multiplicity drops through kerneling}\label{Mult-drops-kerneling}
 This section is devoted to prove that the deflation sequence remains
 constant after a finite index. This will be achieved trough the
 following proposition :
\begin{tm}\label{drop_schur}
Let us suppose that the rank of $Df(\zeta)$ is equal to $r$ and that
$$Df(x):=\left (
\begin{array}{cc}
A(x)&B(x)\\C(x)&D(x)
\end{array}
\right )
$$ 
where $A(\zeta)\in \C^{r\times r}$ is invertible.  Then the
multiplicity of $\zeta$ as root of $K(f)$ is strictly lower than the
multiplicity of $\zeta$ as root of $f$.
\end{tm}
\begin{proof}
If $r=0$ then the system $K(f)$ consists of  all partial derivatives $$\dis \nabla f(x):=\left (\frac{\partial
  f_i(x)}{\partial x_j },\quad 1\le j\le n,\quad 1\le i\le s\right
).$$
 Then, the conclusion follows from the lemma~\ref{drop_K}.

If $r>0$ the system $K(f)$ consists of $f_1,\ldots,f_r$ augmented by
the elements of the schur complement
$D(x)~-~C(x)~A(x)^{-1}~B(x)$.
From the proposition~\ref{schur},  the relations between the lines are
  $$(C(x),D(x))-C(x)A(x)^{-1}(A(x),B(x))=0.$$
It is easy to see that  the system $K(F)=0$
 is equivalent to the following 
  \begin{equation}\label{KF_eq}
  \left (f_1,\ldots, f_r, \nabla f_{i}(x)-\sum_{j=1}^r\lambda_{ij}(x)\nabla f_j(x) =0,\quad i=r+1:s\right )=0,
  \end{equation}
  with $( \lambda_{ij}(x)):=\left (C(x)A(x)^{-1} \right )^T$.
\\ 
From the implicit function theorem, we know that there exists
a local isomorphism $\Phi$ such that 
$$x_{1:r}-\zeta_{1:r}=f_{1:r}\circ \Phi.$$
By substitution of $x_{1:r}-\zeta_{1:r}$ in $f=0$ we obtain the system 
\begin{equation}\label{KF_eq1}
(x_1-\zeta_1,\ldots, x_r-\zeta_r,\, f_{r+1:s}\circ\Phi )=0.
\end{equation}
We remark that the multiplicity of the root $\zeta$ has not changed. The ideal generated by $f_{r+1:s}\circ\Phi$ only contains the monomials $x_i-\zeta_i$, $i=r+1:n$.
On the another hand the multiplicity of $\zeta$ as root of system~(\ref{KF_eq1}) has not changed~: it is also  the multiplicity 
of $\zeta_{r+1:n}$ as root of system $f_{r+1:s}\circ\Phi$.
Moreover,  the multiplicity
of $\zeta$ as root of the system~(\ref{KF_eq}) is equal to the multiplicity of $\zeta_{r+1:n}$ as root of the system
$\nabla( f_{r+1:s}\circ\Phi)$.
We now apply the lemma~\ref{drop_K} to the system $f_{r+1:s}\circ\Phi$ to deduce that the multiplicity drops. We are done.
\end{proof} 
\begin{pp}\label{schur}
Let $\dis M=\left (\begin{array}{cc}
A&B \\C&D  
  \end{array}\right )\in \C^ {s\times n}$ of rank $r$ 
  where $A\in\C^{r\times r}$ is invertible. Then the relations between
  the lines (respectively the columns) of $M$ are given by
  $$ (C,D)-CA^{-1}(A,B)=0,\quad
   (\textrm{respectively $\dis  \left (\begin{array}{c}
B \\D  
  \end{array} \right )-\left (\begin{array}{c}
A \\C  
  \end{array} \right )A^{-1}B=0$} ).$$ 
\end{pp}
\begin{proof} The proposition follows from the equivalence: 
\\ $(C,D)-CA^{-1}(A,B)=0$ and $\dis \left (\begin{array}{c} B \\D
  \end{array} \right )-\left (\begin{array}{c}
A \\C  
  \end{array} \right )A^{-1}B=0$
  iff $D-CA^{-1}B=~0$.  Since the rank of matrix $M$ is equal to $r$,
  this is classically equivalent to $\schur (M)~=~0$.
\end{proof}
\begin{defi}
The valuation of an analytic system $f=(f_1,\ldots,f_s)$ at $\zeta$ is the 
minimum of the valuation of $f_i$'s at $\zeta$.
\end{defi}
\begin{req}
A generator of $I\C\{x-\zeta\}$ of minimal valuation among others
generators can always be taken
as one of the generator of a (minimal) standard basis.
\end{req}

This is a consequence of a fundamental property of local orderings: the
valuation of a sum is larger than the valuation of any of the
summands.\\

In the case where the construction of a standard basis of
$I\C\{x-\zeta\}$ starts from a given set of \emph{polynomial} generators,
the goal can be achieved e.g. through the original Mora's tangent cone
algorithm, by successive $S$-polynomials (and reductions which are
particular cases of them). The valuation can only increase through
these operations, which forbids to reduce $S(f,g)$ by $f$ (or $g$ by
the way).
\begin{lm}\label{drop_K} 
  Let $\dis \nabla f(x):=\left (\frac{\partial f_i(x)}{\partial x_j
  },\quad 1\le j\le n,\quad 1\le i\le s\right )$.  Let us suppose that
  $\zeta$ is an isolated root of $f$ and $\nabla f$.  Then the
  multiplicity of $\zeta$ as root of $\nabla f$ is strictly lower than the
  multiplicity of $\zeta$ as root of $f=0.$
\end{lm}
\begin{proof}
Let us take one of the $f_k$'s, say $f_i$, of minimal valuation
at $\zeta$.  This
valuation is greater than $2$.  There exists an index $j$ such that
the leading term $\dis \frac{\partial f_i(x)}{\partial x_j }$ is not
in the ideal generated by $f$. The conclusion follows.
\end{proof}
\begin{lm}\label{drop_dfl}
Let $p$ the valuation of $f$ at $\zeta$. Let us consider the following system 
$$\dis
D^{p-1}f(x)~:~=~\left (\frac{\partial^{|\alpha|}f_i(x)}{\partial x^\alpha
},\, |\alpha|= p-1,\, 1\le i\le s\right ).$$  Let us assume that $p\ge
2$ and that the rank of $D^pf(\zeta)$ is equal to $r$.
 Then the multiplicity of $\zeta$ as root of $ D^{p-1}f(x)=0$ is strictly
 lower than the multiplicity of $\zeta$ as root of $f=0.$ More precisely
 the multiplicity of the root $\zeta$ drops by at least $p^r$.
\end{lm}
\begin{proof}
Since the valuation $p\ge 2$ then $\dis f(x)=\sum_{k\ge
  p}\frac{1}{k!}D^kf(\zeta)(x-\zeta)^k$ with $D^pf(\zeta)\ne 0$.  The monomials of
$LT(f)$ are of type $(x-\zeta)^\alpha$ with $|\alpha|\ge p\ge 2$. Hence
the number of standard monomials of $\dis \C\{x-\zeta\}/LT(f)$ is bounded
below by $p^n.$ Since the rank of the derivative of $D^{p-1}f(x)$ at
$\zeta$ is $r>0$, we can suppose without loss in generality that
$x_1-\zeta_1,\ldots, x_r-\zeta_r$ are in the ideal $LT(D^{p-1}f(x)
)$. Consequently the number of standard monomials dropped by at least
$p^r$.
\end{proof}
\section{ Quantitative version of Rouch\'es theorem in the regular case}
In this section we consider as previously $\omega\in \C^n$ and the set
$\A2(\omega, R_{\omega})$. For $x\in B(\omega,R_\omega)$
we introduce the quantities
\begin{eqnarray}
&\dis \beta(f,x)=\|Df(x)^{-1}f(x)\|\label{eq_betafx}\ncla{$\beta(f,x)$}
\\\nonumber \\
&\dis\kappa_{x}=\max\left (\, 1,
\,\frac{R_{\omega}(n+1)}{R_{\omega}^2-\rho_{x}^2} \, \right )\label{eq_kappafx}
\ncla{$\kappa_{x_0}$}
\\\nonumber \\
&\dis
\gamma(f,x)=\max \left (1, 
\,\frac{\|f\|\, \|Df(x)^{-1}\|\,R_{\omega}\,\kappa_{x}}{(R_{\omega}^2-\rho_{x}^2)^\frac{n+1}{2}} \,
 \right ) \label{eq_gammafx}\ncla{$\gamma(f,x)$}
\\\nonumber \\
& \dis \alpha(f,x)=\beta(f,x)\,\kappa_{x}\label{eq_alphafx}
\ncla{$\alpha(f,x)$}
\end{eqnarray}

\begin{tm}\label{rouche_regular}($\alpha$-Theorem).
Let $R_{\omega}>0$, $x_0\in B(\omega, R_{\omega})$, 
and $f=(f_1,\ldots,f_n) \in
(\A2(\omega, R_{\omega}))^n$.
  Let us note  $\alpha$, $\beta$, $\gamma$, $\kappa$ for
  $\alpha(f,x_0)$, $\beta (f,x_0)$, $\gamma(f,x_0)$, $\kappa_{x_0}$ 
respectively defined in ~(\ref{eq_alphafx}), ~(\ref{eq_betafx}), ~(\ref{eq_gammafx})
and ~(\ref{eq_kappafx}).
\\
Let us  suppose  that
$$\alpha< 2\gamma+1-\sqrt{(2\gamma+1)^2-1}.$$
Then for all $\theta>0$ such that $B(x_0,\theta)\subset B(\omega,R_{\omega})$ and
$$
\frac{ \alpha+1 -\sqrt{(\alpha+1)^2-4\alpha(\gamma+1)} }{2( \gamma +1)}< 
u:=\kappa\theta < \frac{ 1}{\gamma+1 }$$
    $f$ has only one root in the ball $B(x_0,\theta)$. 
  \end{tm}
Before proving this theorem we need the following proposition.
\begin{pp}\label{Berg_F_DFk_w}
For all $f\in \A2(\zeta,R_{\omega})^s$ we have
   $$\dis\forall k\ge 0,\quad  \frac{1}{k!}\|D^k{f}(x_0)\|\le
||f||\frac{(n+1)^k R_{\omega}^{1+k}}{\left (R_{\omega}^2-\rho_{x_0}^2\right )^{\frac{n+1}{2}+k}}.$$
\end{pp}
\begin{proof}
It is enough to use the inequality
$$ \frac{(n+1)\ldots (n+k)}{k!}\le (n+1)^k$$
  in the proposition~\ref{Berg_F_DFK}.
\end{proof}
We are now ready to begin the proof of the theorem.
\begin{proof}
We let $Df(x_0)^{-1}f(x)=Df(x_0)^{-1}f(x_0)+g(x)$
with
 $$\dis g(x)=x-x_0+\sum_{k\ge 2}\frac{1}{k!}
Df(x_0)^{-1}D^kf(x_0)(x-x_0)^k.$$
We first remark that for all $x\in\C^n$ such that $\dis \|x-x_0\|=\theta$ we have
  \begin{eqnarray}\label{rouche_gx}
  \|g(x)\| &\ge &
  \|x-x_0 \|-\sum_{k\ge 2}\frac{1}{k!}
\|Df(x_0)^{-1} D^kf(x_0)\|\, \|x-x_0 \|^k\nonumber
\\
&
\ge &  \theta-\frac{\|f|\| \,\|Df(x_0)^{-1} \,  R_{\omega}}{(R_{\omega}^2-\rho_{x_0}^2)^\frac{n+1}{2}} 
\sum_{k\ge 2} \left (\frac{(n+1)R_{\omega}\,\theta}  
{R_{\omega}^2 -\rho_{x_0}^2}
  \right )^k
\quad \textsf{from proposition~\ref{Berg_F_DFk_w}}\nonumber
\\
&
\ge
&
\frac{u}{\kappa} - \frac{\gamma}{\kappa}\sum_{k\ge 2} u^k\nonumber\\
  &
  \ge
  &
\frac{1}{\kappa}\left (u-\gamma \frac{u^2}{1-u}\right ).
   \end{eqnarray}
The Rouché's theorem states that the analytic functions 
$Df(x_0)^{-1}f(x)$
and $g(x)$
have the same number of roots, each one  counting with the respective 
multiplicity, in the ball $B(x_0,\theta)$ if the inequality
$$\|Df(x_0)^{-1}f(x)-g(x)\|< \|g(x)\|$$
holds for all $x\in \partial B(x_0,\theta)$.
Let us first prove that $x_0$ is the only root of $g(x)$
in the ball $\dis B\left (x_0,\frac{1}{ \kappa( \gamma +1)}\right )$. In 
fact let $y\ne x_0$ be a root of $g(x)$ in
the ball $B(\omega,R_{\omega})$.
Let $v=\kappa \|y-x_0\|$. If $v\ge 1$ then $\dis \|y-x_0\|\ge 
1/\kappa>\frac{1}{\kappa(\gamma +1)}$. From the assumption we know that  $\dis \frac{1}{\kappa(\gamma +1)}\ge \theta$.
    In this case we conclude that  $y\notin B(x_0,\theta)$.
Otherwise $v<1$.
    We deduce from the inequality~(\ref{rouche_gx}) that
   $$\|g(y)\|=0\ge \frac{1}{\kappa} \left (v -\frac{\gamma  v^2}{1- v}\right 
). $$
   Hence $\dis\frac{1}{\gamma +1}\le v$. 
From the assumption on $\theta$, we then deduce that the distance between two distinct roots is bounded 
 from below by 
 $$\dis \|y-x_0\|\ge\frac{1}{\kappa (\gamma+1 )}>\theta. $$
   We then have proved that $x_0$ is the only one root of $g(x)$ in the 
ball  $\dis  B\left (x_0,  \frac{1 }{    \kappa (\gamma+1 )} \right)$.
   \\
Now, let $x\in B(\omega,R_{\omega})$ be such that $\dis \|x-x_0 
\|=\theta=\frac{u}{\kappa}$.
Then $B(x_0,\theta)\subset B(\omega, 
R_{\omega})$.  Always from the inequality~(\ref{rouche_gx}) we deduce that   
the inequality
\begin{equation}\label{eq_rouche}
\beta:=\|Df(x_0)^{-1}f(x_0)\|< \frac{ 1 }{\kappa } \left (u 
-\frac{\gamma u^2}{1- u}\right )
\end{equation}
implies $\|Df(x_0)^{-1}f(x)-g(x)\|< \|g(x)\|$ on the
boundary of the ball $B(x_0,\theta)$.
Since $\alpha=\beta\kappa$, this is satisfied if the numerator
$$(\gamma+1)u^2-(\alpha+1)u+\alpha$$
of the previous expression~(\ref{eq_rouche}) is strictly negative. Then it is easy to see 
that under the condition
$$\alpha:=\beta\kappa< 2\gamma+1-\sqrt{(2\gamma+1)^2-1}$$
the trinomial $(\gamma+1)u^2-(\alpha+1)u+\alpha$
has two roots equal to
\\ $\dis \frac{ \alpha+1 \pm\sqrt{(\alpha+1)^2-4\alpha(\gamma+1)} }{2( 
\gamma+1)}$. 
Hence
for all $\theta$ such that
$$ \frac{ \alpha+1 -\sqrt{(\alpha+1)^2-4\alpha(\gamma+1)} }{2( 
\gamma+1)}<  u:=\kappa\theta<\frac{ 1}{\gamma+1 }$$
we have $(\gamma+1)u^2-(\alpha+1)u+\alpha<0$.  Then
the inequality ~(\ref{eq_rouche}) is satisfied  and the system
$f$ has only one root in the ball $ \dis B(x_0,\theta)$. The theorem 
follows.
\end{proof}
\section{A new $\gamma$-theorem}\label{NS_gamma_sec}
Let $f=(f_1,\ldots , f_n)$ be an analytic  system which is regular at a 
root $\zeta$ .
The radius of the ball in which the Newton sequence  converges 
quadratically towards a regular root $\zeta$ is controlled by
  the following quantity
  $$\gamma(f,\zeta)=\sup_{k\ge 2}\left (
\frac{1}{k!}\|Df(\zeta)^{-1}D^kf(\zeta)\| \right )^{\frac{1}{k-1}}.$$
    More precisely we have the following
  result named $\gamma$-theorem.
  \begin{tm} ($\gamma$-Theorem of~\cite{BCSS98}). Let $f(x)$ an analytic system and $\zeta$ a
    regular root of $f(x)$. Let $\dis
    R_{\omega}=\frac{3-\sqrt{7}}{2\gamma(f,\zeta)}$. Then for all $x_0\in 
B(\zeta,R_{\omega})$
    the Newton sequence
    $$x_{k+1}=x_k-Df(x_k)^{-1}f(x_k),\quad k\ge 0,$$
     converges quadratically
    towards $\zeta$.
  \end{tm}
  Taking in account the Bergman kernel to reproduce the analytic 
functions  we are going to prove a new version of $\gamma$-theorem for 
analytical regular systems.
\begin{tm}\label{new_gamma_tm}($\gamma$-Theorem).
Let $\zeta$ a regular root of an analytic system $f=(f_1,\ldots,f_n)\in \A2( 
\omega ,R_{\omega})^n$.
Let us note $\gamma$ and $\kappa$ for $\gamma(f,\zeta)$ and  $\kappa_\zeta$
respectively defined in  ~(\ref{eq_gammafx}), ~(\ref{eq_kappafx}).  
Then, for all $x$ be such that
$$   \dis u:=\kappa \,\|x-\zeta\|
< \frac{2\gamma+1-\sqrt{ 4\gamma^2+3\gamma}}{\gamma+1}$$ the 
Newton sequence
$$x_0=x,\quad x_{k+1}=N_f(x_k),\quad k\ge 0,$$
   converges quadratically towards $\zeta$. More precisely
$$\|x_k-\zeta|\le \left ( \frac{1}{2}\right )^{2^k-1}\,\|x-\zeta\|,\quad k\ge 0.$$
\end{tm}
\begin{proof} We use the proposition~\ref{new_gamma_pp}
below to prove by induction the result. The scheme of the proof is 
classical and can be found
   for instance in ~\cite{BCSS98} page 158.
 The assumption $\dis   u<  \frac{2\gamma+1-\sqrt{ 
4\gamma^2+3\gamma}}{\gamma+1}$  implies that
$\dis   \frac{\gamma u}{(1+\gamma)(1-u)^2-\gamma}\le~\frac{1}{2 }$, that is a 
sufficient  condition for the quadratic convergence of the Newton 
sequence with  ratio $\dis \frac{1}{2}$.
\end{proof}
\begin{pp}\label{new_gamma_pp}
With the notations of the theorem~\ref{new_gamma_tm} we have:
\begin{itemize}
\item[1--]   For all $x$ satisfying $\dis u<
1-\sqrt{\frac{\gamma}{1+\gamma }}$,  $Df(x)$ is
invertible. Moreover we have
$$\dis \|Df(x)^{-1}Df(\zeta)\|\le \frac{(1-u)^2}{(1+\gamma)\, 
(1-u)^2-\gamma}.$$
\\
\item[2--] $\dis\| Df(\zeta)^{-1}\left (Df(x)(x-\zeta)-f(x)\right )\|\le
\frac{\gamma u^2}{(1-u)^2}
$.
\\
\\
\item[3--] $\dis \|N_f(x)-\zeta\|\le \frac{\gamma u^2}{ (1+\gamma)\, 
(1-u)^2-\gamma}.$
\end{itemize}
\end{pp}
\begin{proof}\quad
\\
{\hspace{-5cm}\begin{itemize}
\item[1--]
We write
  $$\dis Df(\zeta)^{-1}Df(x)-I=\sum_{k\ge 1}
\binomial{k+1}{k}Df(\zeta)^{-1}\frac{D^{k+1}f(\zeta)}{(k+1)!}(x-\zeta)^k.$$
Using  proposition~\ref{Berg_F_DFk_w},
 \begin{align*}
  \dis \frac{1}{(k+1)!}\|D^{k+1}f(\zeta)\|\,\|Df(\zeta)^{-1}\|
  &\le \frac{||f||\,\|Df(\zeta)^{-1}\|\,(n+1)^{k+1} \,R_{\omega}^{2+k}}{\left (R_{\omega}^2-\rho_{x_0}^2\right )^{\frac{n+1}{2}+k+1}}
  \\\\
&\le
\frac{\|f\|\,\|Df(\zeta)^{-1}\|\,R_{\omega}\,
\kappa^{k+1}}{\left (R_{\omega}^2-\rho_\zeta^2\right 
)^{\frac{n+1}{2}}}
\\ \\
& \le \gamma\kappa^{k },
\end{align*}
we have
  \begin{align*}
  \|Df(\zeta)^{-1}Df(x)-I\|&\le \gamma \sum_{k\ge 1}
\binomial{k+1}{k}\, \left (\kappa\,||x-\zeta\| \right )^k
\\
&\le
\gamma\left (\frac{1}{(1-u)^2}-1\right )
\end{align*}
with $\dis u=\kappa\|x-\zeta\| .$
 From this point estimate and thanks the classical Von Neumann lemma, see for instance \cite{Kato13} page 30,
 the item 1 follows easily.
\item[2--] We have
$\dis Df(x)(x-\zeta)-f(x)=\sum_{k\ge 2}(k-1)\frac{1}{k!}
D^kf(\zeta)(x-\zeta)^k$. Hence, using more the proposition~\ref{Berg_F_DFk_w}  
we get from a straightforward calculation
\begin{align*}
\| Df(\zeta)^{-1}\left (Df(x)(x-\zeta)-f(x)\right )\|
&\le \gamma
\sum_{k\ge 2}(k-1)\left (\kappa\|x-\zeta\|\right )^k
\\
&\le
\frac{\gamma u^2}{(1-u)^2}.
\end{align*}
This proves the item 2.
\item[3--] We write
$$N_f(x)-\zeta=Df(x)^{-1}Df(\zeta)\, Df(\zeta)^{-1}(Df(x)(x-\zeta)-f(x)).$$
Using the items 1 and 2, we get the result.
\end{itemize}
  }
\end{proof}
  From the theorem~\ref{new_gamma_tm} we can state  :
\begin{tm}\label{NS_gamma_tm}($\gamma$-theorem).
Let $f \in \A2(\omega,R_{\omega})^s$ and $\zeta\in B(\omega, R_{\omega})$ such that
$f(\zeta)=0$. Let  us suppose there exists a index  $\ell$ be such that
\begin{itemize}
\item[1--] For all $0\ge k<\ell$ each element $F_k=K(F_{k-1})$ satisfies $F_k(\zeta)=0$ and $\rank(DF_k(\zeta))~<~n$.
\item[2--] The assumptions of $\gamma$-theorem~\ref{new_gamma_tm}
hold for the system $F_\ell$ at $\zeta$. 
\end{itemize} 
Then, for all $x$ be such that
$$ u:=\kappa\,\|x-\zeta\|\,< \frac{2\gamma+1-\sqrt{ 4\gamma^2+3\gamma}}{\gamma+1}$$ the 
Newton sequence, computed thanks to the Table~\ref{Ndflx0},
$$x_0=x,\quad x_{k+1}=N_{\dfl(f)}(x_k),\quad k\ge 0,$$
   converges quadratically towards $\zeta$.
\end{tm}
Also an existence result of a singular solution follows from the 
theorem~\ref{rouche_regular}.
  \begin{tm}\label{rouche_deflation}
  Let $f \in \A2(\omega,R_{\omega})^s$ and $x_0\in B(\omega,R_{\omega})$. Let us suppose 
that there exists a
  deflation sequence $\dis (F_k)_{0\le k \le \ell}$
of thickness $\ell$ at $x_0$.
More precisely
\begin{itemize}
\item[1--] For all $0\ge k<\ell$ each element $F_k=K(F_{k-1})$ satisfies 
\begin{itemize}
\item[1.1--] $\dis \|F_k(x_0)\|\le  \eta_k:=\frac{2\alpha_0}{(n+1)(n+2)(R_{x_0}+\|F_k\|)R_{x_0}^{n-2}}$.
\item[1.2--] $DF_k(x_0)$ has a $\varepsilon_k$ numerical rank stictly less than $n$ where $\varepsilon_k$ is the $\varepsilon $ number of the line 
$6$ of the Table $1$.
\end{itemize}
\item[2--] The assumptions of $\alpha$-theorem~\ref{rouche_regular}
hold for the system $F_l$ at $x_0$. 
\end{itemize} 
Then     $f$ has only one root in the ball $B(x_0,\theta)$
where $\theta$ is defined in $\alpha-$theorem~\ref{rouche_regular}.
  \end{tm}

\section{Example}
Let us give an example to illustrate the exact and numerical algorithm, 
by considering
$f(x,y)=(f_1(x,y),f_2(x,y))$ with
$$f_1(x,y)=x^3/3+y^2x+x^2+2yx+y^2,\quad 
f_2(x,y)=x^2y-y^2x+x^2+2yx+y^2.$$
The root $(0,0)$ has multiplicity $6$.
\subsection{Exact computations}
   We have
$$ Df(x,y)=\left (\begin{array}{cc}
x^2+y^ 2+2x+2y&2xy+2x+2y\\2xy-y^2+2x+2y&x^2-2xy+2x+2y
\end{array}\right ).$$
The rank of the Jacobian matrix is $0$ at $(0,0)$. Hence kerneling
consists just to add to the input system the gradients of $f_1$ and
$f_2$:
$$F_1 = K(f)=(x^2+y^2+2x+2y, 2xy+2x+2y, 2xy-y^2+2x+2y,
x^2-2xy+2x+2y).$$
The four last lines of Jacobian matrix of $K(f)$ are:
$$\left (\begin{array}{cc}
2x+2&2y+2\\2y+2&2x+2\\
2y+2&2x-2y+2\\
2x-2y+2&-2x+2
\end{array}\right ).$$
The rank at $0,0$ of the matrix $\left (\begin{array}{cc}
2&2\\2&2\\2&2 \\2&2
\end{array}\right )$ is one, as the rank of the Jacobian of $F_1$ at
$(0,0)$.\\
The Schur complement of $DF_1(x,y)$ associated to $2x+2$ is
$$
\schur(DF_1(x,y))=\frac{2}{x+1}\left (\begin{array}{c}
2x-2y+x^2-y^2
\\
2x-3y+x^2-xy-y^2
\\
-x-x^2-xy+y^2.
\end{array}\right )
$$
Then we can easily check  that the system $F_2=(f_1,vec(\schur(DF_1(x,y))$ is a regular system equivalent at $(0,0)$ to $f$.
Let us remark also the truncated system of $F_2$ up to the order $1$ namely
$$(x+y,\, x-y,\, 2x-3y,\, x)$$
 is a regular system equivalent at $(0,0)$ to $f$.
  \subsection{Numerical computations}
We give the behaviour of the deflation sequence.
    \begin{itemize}
       \item[1--] The  initial point                        
$(x_0,y_0)=(-0.01, 0.02)$.
    \item[2--] The system   :
\begin{align*}
f= \left (
  \begin{array}{ccc}
  1/3\,{x}^{3}+{y}^{2}x+{x}^{2}+2\,xy+{y}^{2}
\\
{x}^{2}y-y^2x+{x}^{2}+2\,xy+{y}^{2}
  \end{array}
       \right )
\end{align*}
\item[3--] The ball $B(x_0, R_{x_0}):=B(x_0,1/4)$
  \item[4--]
Truncated expansion series of the system $F_0=f(x+x_0,y+y_0)$ up to the 
order $3$.
\begin{align*}
F_0=\left (
  \begin{array}{ccc} 0.0000957+ 0.0205\,x+ 0.0196\,y+ 0.990\,{x}^{2}+ 
2.04\,xy+ 0.99\,{y}^{2}+ 0.333\,{x}^{3}+{y}^{2}x
  \\
  0.000106+ 0.0205\,y+ 0.0192\,x+ 1.94\,xy+ 1.02\,{x}^{2}+1.01{y}^{2}+{x}^{
2}y-y^2x
  \end{array}
       \right )
\end{align*}
     \item[5--] Evaluation of $F_0$ at $(0,0)$ : $(0.0000956666667, 
0.000106)$.
     \item[6--] We successively have
     $\|F_0\|=8\times 10^{-4}$,
        \\$  \dis \eta=\frac{2\alpha_0}{12(R_{x_0}+\|F_0\|)R_{x_0}^{n-2}}=0.086 ~> 
~\|F_0(0,0)\|~=~0.000106.$
       \item[7--] Jacobian of $F_0$  at $(0,0)$:
       $
   DF_0(0,0)=\left( \begin {array}{cc}  0.02050000000& 0.0196\\ 
\noalign{\medskip}
  0.0196& 0.0205\end {array} \right )
     $.
     The singular values of this jacobian are $0.039$ and $0.0011$.
     This jacobian has  a $\varepsilon_0=0.086$-rank equal to $0$.
    \item[8--] Kerneling of $F_0$ at $(0,0)$ :
\begin{align*}
F_1=K(F_0)=\left (\begin{array}{ccc}
  0.0205+ 1.98\,x+ 2.04\,y+ 1.0\,{x}^{2}+{y}^{2}
  \\0.0196+ 2.04\,x+ 1.98\,y+ 2.0\,xy
  \\0.0192+1.94y+2.04x+2xy-y^2
  \\  0.0205+ 1.94\,x+ 2.02\,y+{x}^{2}-2xy
\end{array}
\right )
\end{align*}
     \item[9--] Evaluation of $F_1$ at $(0,0)$ : $F_1(0,0)=(0.0205, 
0.0196,0.0192, 0.0205)$.
     We have $\|F_1\|=0.1$ and
      \\$\dis \eta=\frac{2\alpha_0}{12(R_{x_0}+\|F_1\|)R_{x_0}^{n-2}}=0.062> 
\|F_1(0,0)\|=0.034.$
        \item[10--]       Jacobian matrix of $F_1$ and its evaluation at 
$(0,0)$:
       $$  DF_1(x,y)=\left( \begin {array}{cc}  1.98+ 2.0\,x& 2.04+2\,y
\\ \noalign{\medskip} 2.04+ 2\,y& 1.98+ 2\,x\\ 
\noalign{\medskip} 2.04+ 2\,y& 1.94+ 2\,x-2y
\\\noalign{\medskip}
  1.94+2\,x-2y& 2.02-2x\end {array} \right )
  \quad\quad
     DF_1(0,0)=\left ( \begin {array}{cc}  1.98& 2.04\\ 
\noalign{\medskip}
  2.04& 1.98\\
   \noalign{\medskip} 2.04& 1.94\\ \noalign{\medskip} 1.94& 2.02\end {array} \right )
     $$
     The singular values of  $DF_1(0,0)$ are $5.6$ and $0.1$ and
  its $\varepsilon_1=0.21$-rank  is one.
\item[11--] Kerneling of $DF_1(x,y)$. We compute the truncated series at 
the order one in $(0,0)$ of each element of the Schur
        complement of $DF_1(x,y)$ associated to
         $1.98+ 2.0\,x$. We obtain
\begin{align*}
F_2:=K(F_1)=
\left(\begin{array}{cccc}
  0.0205+ 1.98\,x+ 2.04\,y
\\- 0.12+ 4.12\,x- 4.12\,y
\\-0.16+4.12x-6.12y
\\ 0.021-2.04x+0.1y
   \end{array}
               \right )
\end{align*}
\item[12--]  Regular system from $F_2$ at $(0,0)$. The singular values 
of $DF_2(0,0)$ are $9.46$ and $3.32$ and $DF_2(0,0)$ has $\varepsilon_2=3.32$-rank equal to $2$.
\item[13--] If we consider
     \begin{align*}
  dfl(f)=
  \left(\begin{array}{cccc}
  0.0205+ 1.98\,x+ 2.04\,y
\\- 0.121+ 4.123\,x- 4.121\,y
\end{array}
  \right )
\end{align*}
we find that the iterate of $(x_0,y_0)=(-0.01,0.02)$ is by the singular 
Newton operator is $(-0.0001017,0.00034)$. This illustrates  the
manifestation of a quadratic convergence.
   \end{itemize}
    We show below  quadratic convergence obtained thanks to the algorithm 
singular Newton.
\begin{align*}
&[- 0.01, 0.02]
\\\\
&
[- 0.00010175, 0.000343]
\\\\
&
[- 1.7\times 10^{-8 },  8.1\times 10^{-8 }]
\\\\
&
[-{ 7.15\times 10^{-16}},{ 4.2\times 10^{-15}}]
\\\\
&
[-{ 1.5\times 10^{-30}},{ 1.06\times 10^{-29}}]
\\\\
&
[-{ 7.9\times 10^{-60}},{ 6.55\times 10^{-59}}]
\\\\
&
[-{ 2.6\times 10^{-118}},{ 2.4\times 10^{-117}}]
\end{align*}
\subsection{Illustration of the theorems~\ref{new_gamma_tm} and
~\ref{rouche_regular}         }
This is given by the table below
  $$
\hspace{-2cm} \begin{array}{|c|c|c|c|c|c|c|c|}
\hline
&&&&&&&\\
&\beta&\kappa&\gamma&\alpha&\scr 
2\gamma+1-\sqrt{(2\gamma+1)^2-1}&\rho_{x_0}=
\frac{2\gamma+1-\sqrt{(2\gamma+1)^2-1}}{2\kappa(\gamma+1)}&
\frac{2\gamma+1-\sqrt{ 4\gamma^2+3\gamma}}{\gamma+1}
\\\hline
(-0.001,0.002) & 0.00019  &6 &13.06&0.0119&0.018&0.00051&
\\
\hline
(0,0)& 0&6&13.012&0 &&&0.0031
\\
\hline
  \end{array}
  $$

\bibliographystyle{acm}
 \bibliography{vc.bib}
 \end{document}